\documentclass{amsart}

\usepackage{amssymb}
\usepackage[cp1250]{inputenc}
\usepackage{picture}
\usepackage{color}
\usepackage{tikz}
\usepackage{hyperref}
\usepackage{tikz}
\usetikzlibrary{matrix}

\newtheorem{theorem}{Theorem}[section]
\newtheorem{lemma}[theorem]{Lemma}
\newtheorem{example}[theorem]{Example}

\newtheorem{proposition}[theorem]{Proposition}
\newtheorem{problem}[theorem]{Problem}
\newtheorem{corollary}[theorem]{Corollary}

\newtheorem{definition}[theorem]{Definition}

\newcommand{\N}{\mathbb N}

\newcommand{\R}{\mathbb R}

\author{Szymon G\l \c ab}
\address{Institute of Mathematics, \L \'od\'z University of Technology,
W\'olcza\'nska 215, 93-005 \L \'od\'z, Poland}
\email {szymon.glab@p.lodz.pl}

\author{Jacek Marchwicki}
\address{
Chair of Complex Analysis,
Faculty of Mathematics and Computer Science,
University of Warmia and Mazury in Olsztyn,
S\l oneczna 54,
10-710 Olsztyn,
Poland}
\email {marchewajaclaw@gmail.com}

\title[Set of uniqueness for Cantorvals]{Set of uniqueness for Cantorvals }


\begin{document}

\begin{abstract}
The main result is that the celebrated Guthrie-Nymann's Cantorval has comeager set of uniqueness. On the other hand many other Cantorvals have meager set of uniqueness.
\end{abstract}

\maketitle 

\section{Introduction}
A Cantorval is a subset of the real line which shares properties of the Cantor set as well as of an interval. Its construction reminds that of Cantor set, but it is regularly closed, that means it is the closure of its interior. To construct a Cantorval one mimics the construction of the ternary Cantor set. The difference is that in odd steps one removes middle intervals, while in even steps remains them. This is the way to produce a symmetric Cantorval of an M-Cantorval, which we call briefly a Cantorval.

Cantorvals appear naturally when one consider arithmetic sum of two Cantor sets, see \cite{AC,MO,MMR}. It turns out that the range of a purely atomic finite measure, or equivalently a set of subsums of an absolutely convergent series, can be a Cantorval. In this paper we study the sets of uniqueness of Cantorvals in their measure or series representation. It is worth to mention that not every Cantorval has such representation, \cite{BFGPS}. Having a purely atomic probabilistic measure $\mu$, a real number $t\in[0,1]$ is unique, if there is exactly one event $A$ with $\mu(A)=t$. Set of uniqueness of $\text{rng}(\mu)\subseteq[0,1]$ consists of all unique numbers.   

We show that the set of uniqueness has an empty interior. The main result is that the celebrated Guthrie-Nymann's Cantorval has comeager set of uniqueness. In the opposite to unique values are those which appear continuum many times, called $\mathfrak{c}$-points. We observe that any Cantorval can be enlarged to a Cantorval with $\mathfrak{c}$-points. We study minimal Cantorvals, that is Cantovals which cannot be essentially shrunk. Finally, we characterize achievement sets generated by semi-fast convergent sequence which are slim. An achievement set is slim if it has a representation without $\mathfrak{c}$-points.   

\section{Preliminaries}
By $E(x_n)$ we denote the set of all subsums of series $\sum_{n=1}^{\infty}x_n$, that is $$E(x_n)=\{\sum_{n=1}^{\infty}\varepsilon_nx_n : (\varepsilon_n)\in\{0,1\}^{\mathbb{N}}\}=\{\sum_{n\in A}x_n : A\subset\mathbb{N}\}.$$ Jones in \cite{Jones} called $E(x_n)$ as achievement set of the series $\sum_{n=1}^{\infty}x_n$. In a special interest will be the set $U(x_n)\subset E(x_n)$ of all of that points which has unique representation, that is $x\in U(x_n)$ if there is only one sequence $(\varepsilon_n)\in\{0,1\}^{\mathbb{N}}$ (respectively only one set $A\subset\mathbb{N}$) such that $x=\sum_{n=1}^{\infty}\varepsilon_nx_n$ ($x=\sum_{n\in A}x_n$). We will also call any element $x\in U(x_n)$ as unique.
Let $\sum_{n=1}^{\infty}x_n$ be absolutely summable. The function $\mu : \{0,1\}^\N\ni(\varepsilon_n)\mapsto\sum_{n=1}^\infty\varepsilon_nx_n$ is continuous mapping from the Cantor space $\{0,1\}^\N$ to the real line. Therefore $\mu^{-1}(t)$ is a closed subset of a Polish space $\{0,1\}^\N$, so the cardinality $\vert\mu^{-1}(t)\vert$ belongs to the set $\{0,1,2,\dots,\omega,\mathfrak{c}\}$ where $\omega$ stands for the first infinite cardinal while $\mathfrak{c}$ for the continuum.
 In the whole paper when we count number of representations for some point $x$ we will often write shortly that $x$ is $1$ - point(respectively $2, 3, \ldots, \omega, \mathfrak{c}$) instead of writing that $x$ has $1$(respectively $2, 3, \ldots, \omega, \mathfrak{c}$) representation. 
In the paper \cite{cardfun} the authors called $ f :E(x_n)\rightarrow \{1,2, 3, \ldots, \omega, \mathfrak{c}\}$ defined as $f(t)=\vert \{ A:  \sum_{n\in A}x_n=t\}\vert$ as the cardinal function. We will also say that $f$ is cardinal function for the sequence $(x_n)$, for the series $\sum_{n=1}^{\infty}x_n$ or for the achievement set $E(x_n)$.
The preimage of the cardinal function $f^{-1}(1)$ (respectively $2, 3, \ldots, \omega, \mathfrak{c}$) are equal to the set of all $1$-points (respectively $2, 3, \ldots, \omega, \mathfrak{c}$).  
Since $E(\vert x_n\vert)=E(x_n)+\sum_{n=1}^{\infty}x_n^{-}$ and $f(x-\sum_{n=1}^{\infty}x_n^{-})=g(x)$ for each $x\in E(\vert x_n\vert)$, where $f$ and $g$ are the cardinal functions for $\sum_{n=1}^{\infty}x_n$ and $\sum_{n=1}^{\infty}\vert x_n\vert$ respectively, we may consider only positive terms.
A simple observation shows that for an absolutely convergent series by rearranging its terms we do not affect the set $E(x_n)$ neither the cardinal function. Hence in the whole paper we will assume that $(x_n)$ is nonincreasing. 
First paper, where achievement set was considered is that of Kakeya, see \cite{Kakeya}. The author proved the following
\begin{theorem}\label{Kakeya}
If $\sum_{n=1}^{\infty}x_n$ is absolutely convergent with infinite many nonzero terms, then 
\begin{enumerate}
\item $E(x_n)$ is a finite union of closed intervals iff $x_k\leq r_k=\sum_{n=k+1}^{\infty}x_n$ for all but finitely many natural $k$
\item $E(x_n)$ is homeomorphic to a Cantor set, if $x_k>r_k=\sum_{n=k+1}^{\infty}x_n$ for all $k\in\mathbb{N}$ 
\end{enumerate} 
\end{theorem}  
Moran in \cite{M1} called series, for which $x_k>r_k$  for all $k\in\mathbb{N}$ as quickly convergent.
Not only is a quick convergence a sufficient condition for $E(x_n)$ to be homeomorphic to the Cantor set, but it also sufficies to $E(x_n)=U(x_n)$. However both conditions $E(x_n)$ being homeomorphic to the Cantor set and what is much more surpring equality $E(x_n)=U(x_n)$ may hold even when  $x_{2n+1}<r_{2n+1}$ for all $n\in\mathbb{N}$, see \cite{cardfun}. 

 On the other hand $x_k\leq r_k$  for all $k\in\mathbb{N}$ if and only if $E(x_n)$ is an interval. Such series are called slowly convergent and the sequence of terms $(x_n)$ is called interval filling. There are many papers dedicated to interval filling sequences, see \cite{DJK}, \cite{DK}, \cite{DKS}. In particular in \cite{DJK} the authors considered series for which $U(x_n)$ is the smallest as possible, that is $U(x_n)=\{0,\sum_{n=1}^{\infty} x_n\}$ and gave a nice and easily verified, sufficient condition for that. Namely the sequence $(x_n)$ should remain interval filling after removing any one of its terms. Such sequences were called in \cite{DJK} as lockers.
 It is worth to mention that there is another paper on particular kind of slowly convergent series. In \cite{Erdos} the authors considered geometric series with ratio $1>q>\frac{1}{2}$. They were mostly interested in the problem how the number of subseries which sum equals  $1$ depends on $q$.

Back to the beginning Kakeya claimed that for an absolutely  convergent with infinite many nonzero terms the set $E(x_n)$ is either a finite union of closed intervals or a set homeomorphic to a Cantor set. It appears that he was wrong and due to Guthrie and Nymann we know that there is one more possible form. 
\begin{theorem}\label{GNCantorval}
For an absolutely convergent series $\sum_{n=1}^{\infty}x_n$ with infinite many nonzero terms, the set $E(x_n)$ is one of the following: a finite sum of closed intervals, homeomorphic to a Cantor set or a Cantorval, that is a set homeomorphic to $E(y_n)$ for $y_{2n-1}=\frac{3}{4^n}$,  $y_{2n}=\frac{2}{4^n}$ for all $n\in\mathbb{N}$.
\end{theorem} 
Theorem \ref{GNCantorval} was first published in \cite{GN}, but the correct proof was given in \cite{NS0}.
The set $E(y_n)$ is called the Guthrie Nymann Cantorval. It is obtained for a series belonging to a multigeometric class, that is of the form $(x_n)=(c_1,c_2,\ldots,c_m;q)=(c_1q,c_2q,\ldots,c_kq,c_1q^2, c_2q^2,\ldots,c_kq^2,c_1q^3,\ldots)$. Using that notion the Guthrie and Nymann's Cantorval can be described as $E(3,2;\frac{1}{4})$. In special interest we will have a set $\Sigma=E(c_1,\ldots,c_m)$, that is $\Sigma=  \{\sum_{n=1}^{m}\varepsilon_nc_n : (\varepsilon_n)\in\{0,1\}^{m}\}$. Then $E(c_1,c_2,\ldots,c_m;q)=\{\sum_{n=1}^{\infty}y_nq^n : (y_n)\in\Sigma^{\infty}\}$. Multigeometric series were studied in \cite{BBFS}, \cite{BBGS} and \cite{BFS}.

Note that the negative answer for Kakeya's hypothesis was obtained before Guthrie and Nymann's paper.
First counterexample was given without proof by Weinstein and Shapiro in \cite{WS}.
In \cite{F} Ferens constructed a purely atomic finite measure $\mu$, and proved that its range is a Cantorval. The theory of achievement sets and purely atomic  finite measure coincide. Indeed we may assume that $\mu$ is defined on $\mathbb{N}$. Then $rng(\mu)=\{\mu(A) : A\subset\mathbb{N}\}=E(x_n)$, where the terms of our series are the values of measure on singletons, that is $x_n=\mu(\{n\})$ for all $n\in\mathbb{N}$. Hence we may say that Ferens observed that $E(7,6,5,4,3;\frac{2}{27})$ is a Cantorval.

Note that by a condition given in \cite{DJK} we may construct plenty of  various Examples of series $\sum_{n=1}^{\infty}x_n$ for which $E(x_n)$ is a closed interval and $U(x_n)$ contains only two points. On the other hand we will show that if $E(x_n)$ is a Cantorval, then $U(x_n)$ is infinite. We use the following Lemma, see \cite{BFPW}. 

\begin{lemma}\label{thirdgaplemma}
(Third Gap Lemma) Suppose that $(a,b)$ is a gap in the range $R$ such that for
any gap $\left( a_{1},b_{1}\right) $ with $b_{1}<a$ we have $b-a>b_{1}-a_{1}$
(in other words $\left( a,b\right) $ is the longest gap from the left). Then for some $k\in \mathbb{N}$  we have $b=x_{k}$  and $a=r_{k}$.
\end{lemma}

There are several consequences of Lemma \ref{thirdgaplemma}. 
\begin{corollary}\label{interval}
If $U(x_n)=\{0,\sum_{n=1}^{\infty}x_n\}$, then $E(x_n)$ is an interval.
\end{corollary}
\begin{corollary}\label{cantorval}
If  $E(x_n)$ is a Cantorval, then $U(x_n)$ is infinite.
\end{corollary}
Note that the role of non-unique element in $(0,\sum_{n=1}^{\infty}x_n)$ is played by $a$ from Lemma \ref{thirdgaplemma}. Indeed since $b=x_k>r_k=a$ it is clear there is only one tail-representation of $a$. On the other hand it may happen that $b=x_k=x_{k+1}$, so $b$ can not be considered as unique element.

\begin{proposition}\label{onlycantor}
Let $F: \{0,1\}^{\mathbb{N}}\rightarrow E(x_n)$ be defined as $F((\varepsilon_n))=\sum_{n=1}^{\infty}\varepsilon_nx_n$. 
\begin{enumerate}
\item If $\sum_{n=1}^{\infty}x_n$ is absolutely convergent, then $F$ is continuous.
\item Moreover if $F$ is $1-1$, then $F$ is homeomorphism. 
\end{enumerate}
\begin{proof}
$(1)$. It is a folklore.
\\$(2)$.  It is well known that continuous injection on a compact set is homeomorphism.
\end{proof}
\end{proposition}

By Proposition \ref{onlycantor} (2) we know that the property of unique representations of all points, that is $E(x_n)=U(x_n)$ is reserved for Cantor sets. 

In \cite{cardfun} the authors showed that if $E(x_n)$ is a finite union of closed intervals, then $U(x_n)$ has an empty interior. Corollary \ref{cantorval} may suggest that if $E(x_n)$ is a Cantorval, then $U(x_n)$ is large. In this paper we will show that it is not true and $U(x_n)$ has an empty interior in case $E(x_n)$ is a Cantorval as well.

\section{$U(x_n)$ has empty interior}

We will use the following simple observation.
\begin{proposition}\label{skonczonezmiany}
Let $k\in\mathbb{N}$. Then $E(x_n)$ is a finite union of closed intervals (a Cantor set, a Cantorval) if and only if $E((x_n)_{n>k})$ is a finite union of closed intervals (a Cantor set, a Cantorval, respectively).
\end{proposition}
The above proposition can be read as follows: removing or adding finitely many terms to our series does not change the type of its achievement set. In \cite{cardfun} the authors asked in Problem 3.14 if it is possible to construct a sequence $(x_n)$ for which $U(x_n)$ contains an interval. They also gave a negative answer for the case when $E(x_n)$ is a finite union of closed intervals. Theorem \ref{Cantorvalnointerval} completes the answer.

\begin{theorem}\label{Cantorvalnointerval}
Let $E=E(x_n)$ contains an interval ($E$ is a Cantorval either finite sum of closed intervals). Then $U=U(x_n)$ has an empty interior. 
\begin{proof}
Suppose that every point of $[a,b]\subset E$ is unique. Firstly we will show that it implies that there exists $\varepsilon>0$ such that every point in the set $E\cap [0,\varepsilon]$ is unique.
\\Suppose that for each $\varepsilon>0$ the set $E\cap [0,\varepsilon]$ contains a non-unique point. Since a set of finite sums $\{\sum_{n=1}^{k} \varepsilon_n x_n : (\varepsilon_n)\in\{0,1\}^k, k\in\mathbb{N}\}$ is dense in $E$ one can find a finite set $A\subset\{1,\ldots,m\}$ such that $\sum_{n\in A} x_n\in (a,b)$. Denote $\delta=b-x$. Let $\alpha$ be any positive number smaller than both $x_m$ and $\delta$, that is $\alpha<\min\{x_m,\delta\}$. One can find a non-unique $y\in E\cap[0,\alpha]$. Hence $y=\sum_{n\in B} x_n=\sum_{n\in C} x_n$ for $B\neq C$. Since $y<x_m$, we get $A\cap B=\emptyset=A\cap C$. Thus $x+y=\sum_{n\in A\cup B} x_n=\sum_{n\in A\cup C} x_n$, so $x+y$ is non-unique. But $a<x<x+y<x+\alpha<x+\delta=x+b-x=b$, which means that $x+y\in [a,b]$ and contradicts with uniqueness of points in $[a,b]$. Hence we are done with the first part of the proof. 
\\We get that $U\supset E\cap [0,\varepsilon]$ for some $\varepsilon>0$. Note that $\sum_{n=k+1}^{\infty}x_n<\varepsilon$ for large enough $k$.
Hence by removing first $k$ terms, we get that $E((x_n)_{n>k})\subset [0,\varepsilon]$. But then by Proposition \ref{skonczonezmiany} we know that $E((x_n)_{n>k})$ is a Cantorval either a finite sum of closed intervals, which contains only unique points. By Proposition \ref{onlycantor} it is not possible.
\end{proof}
\end{theorem}

\section{Guthrie--Nymann's Cantorval revisited}
In the most of the papers, dedicated to counting numbers of representations, the authors considered $E(x_n)$ being an interval, see \cite{Nymann} or \cite{Erdos}. 
The first paper, where the number of digital representations of each point was calculated for Cantorval is \cite{BPW}. The authors considered the Guthrie--Nymann's Cantorval $E(x_n)=E(3,2;\frac{1}{4})$ and showed that each point may have $1$ either $2$ representations. Clearly $E(3,2;\frac{1}{4})=\{\sum_{n=1}^{\infty}\frac{b_n}{4^n} : (b_n)\in\Sigma^{\infty}\}$ for $\Sigma=\{0,2,3,5\}$. Note that the elements of $\Sigma$ are bijectively generated as subsums of $3$ and $2$. Hence the equality
$\sum_{n=1}^{\infty}\frac{b_n}{4^n}=\sum_{n=1}^{\infty}\varepsilon_n x_n$ reads as follows: $b_n=5$ iff $\varepsilon_{2n-1}=\varepsilon_{2n}=1$; $b_n=3$ iff $\varepsilon_{2n-1}=1$ and $\varepsilon_{2n}=0$; $b_n=2$ iff $\varepsilon_{2n-1}=0$ and $\varepsilon_{2n}=1$; $b_n=0$ iff $\varepsilon_{2n-1}=\varepsilon_{2n}=0$.
 It gives a  straightforward conversion between the two ways for writing a subsum, one uses  the sequences $(b_n)$ and the second, which uses the terms $(x_n)$.

The authors of \cite{BPW} proved that a point has two digital representations $\sum_{n=1}^{\infty}\frac{a_n}{4^n}=\sum_{n=1}^{\infty}\frac{b_n}{4^n}$ iff there exists the finite or infinite sequence $n_0<n_1<\ldots$ such that 
\begin{enumerate}
\item $a_k=b_k$ for $k<n_0$;
\item $a_{n_0}=2$ and $b_{n_0}=3$;
\item $a_{n_k}=5$ and $b_{n_k}=0$ for odd $k$;
\item $a_{n_k}=0$ and $b_{n_k}=5$ for even $k>0$;
\item  $a_{i}\in\{3,5\}$ and $a_i-b_{i}=3$, as far as $n_{2k}<i<n_{2k+1}$;
\item  $a_{i}\in\{0,2\}$ and $b_i-a_{i}=3$, as far as $n_{2k+1}<i<n_{2k+2}$.
\end{enumerate} 
Otherwise a point has a unique digital representation. The above conditions should be interpretated as follows:
By $(1)$ we see it does not matter what appears in the begining of digital representation. Since the set of all sums of finite subseries is dense in $E$, we obtain that the set $E_2$ of all $2$ - points is dense in Cantorval. Moreover the authors mentioned in \cite{BPW} that in particular if $a_{n}=2$ and $a_{n+1}=3$ for infinite many $n$'s, then the representation $\sum_{n=1}^{\infty}\frac{a_n}{4^n}$ is unique. Again, since that condtion does not depend on the first finitely many terms, we get that  the set $U=E_1$ of all $1$ - points is dense in Cantorval.
Condition $(2)$ is a consequence of the fact that $\sum_{n=k}^{\infty}\frac{5}{4^n}=\frac{5}{3}\cdot\frac{1}{4^{k-1}}<2\cdot\frac{1}{4^{k-1}}$. It precisely means that if there are two digital representations, their partial sums should be close to each other, namely $\vert\sum_{n=1}^{k}\frac{a_n}{4^n}-\sum_{n=1}^{k}\frac{b_n}{4^n}\vert\leq\frac{1}{4^k}$ for all $k\in\mathbb{N}$, because both $\sum_{n=1}^{k}\frac{a_n}{4^n}$ and $\sum_{n=1}^{k}\frac{b_n}{4^n}$ can be written as fraction with denominator $4^k$.
If this distance is positive for the first time for $k=n_0$, since $\frac{1}{4^k}=\sum_{n=k+1}^{\infty}\frac{3}{4^n}$ and $\frac{5}{4^{k+1}}-\frac{1}{4^k}=\frac{1}{4^{k+1}}$, it has to remain positive for any $k>n_0$ and the last inequality changes into equality. Hence we get conditions $(3), (4), (5)$ and $(6)$ as the only possible ways to keep the distance between partial sums. 

Now we consider Cantorvals for which $E=E_1\cup E_2$ and both sets $E_1$ and $E_2$ are dense. Our method of calculation and argumentation is analogous to that used by authors in \cite{BPW}. All of the considered Cantorvals belongs to a class of Guthrie-Nymann-Jones' Cantorvals, i. e. $E(3,\underbrace{2,2,\ldots,2}_{m - \ \text{times}};q)$.  We ponder the ratio $q=\frac{1}{2m+2}$. Note that in \cite{BFS} the authors showed that for the taken parameters the set $E$ is a Cantorval. Moreover we should guarantee uniqueness in $\Sigma$, so it has to have $2^p$ elements, when $\Sigma$ is generated by $p$ elements. This problem of the uniqueness representation in $\Sigma$ in GNJ-Cantorvals we solve by taking particular numbers of $2$, that is $m=2^r-1$ and combining consecutive $2^w$ elements $2$ into one $2^{w+1}$. That is instead of $E=E(3,2,2,2;\frac{1}{8})$ we consider the same set but obtained for another sequence $E=E(4,3,2;\frac{1}{8})$, instead of  $E(3,2,2,2,2,2,2,2;\frac{1}{16})$ its counterpart $E=E(8,4,3,2;\frac{1}{16})$ and so on. 

\begin{theorem}
Assume that $x\in E(4,3,2;\frac{1}{8})$ has more than one digital representation. There exists the finite or infinite sequence of positive natural numbers $n_0<n_1<\ldots$ and exactly two digital representations $x=\sum_{n=1}^{\infty}\frac{a_n}{8^n}=\sum_{n=1}^{\infty}\frac{b_n}{8^n}$ of $x$ such that 
\begin{enumerate}
\item $a_k=b_k$ for $k<n_0$;
\item $a_{n_0}=2+j$ and $b_{n_0}=3+j$ for some $j\in\{0,1,2,3,4\}$;
\item $a_{n_k}=9$ and $b_{n_k}=0$ for odd $k$;
\item $a_{n_k}=0$ and $b_{n_k}=9$ for even $k>0$;
\item  $a_{i}\in\{7,9\}$ and $a_i-b_{i}=7$, as far as $n_{2k}<i<n_{2k+1}$;
\item  $a_{i}\in\{0,2\}$ and $b_i-a_{i}=7$, as far as $n_{2k+1}<i<n_{2k+2}$.
\end{enumerate}
\begin{proof}
Assume that $x=\sum_{n=1}^{\infty}\frac{a_n}{8^n}=\sum_{n=1}^{\infty}\frac{b_n}{8^n}$ for two different digital representations $(a_n)$ and $(b_n)$ in $\{0,2,3,4,5,6,7,9\}^\N$. Let $n_0$ be the minimal index, where $(a_n)$ and $(b_n)$ differs. That is condition $(1)$ holds and $a_{n_0}\neq b_{n_0}$. Since $\sum_{n=n_0+1}^{\infty}\frac{9}{8^n}=\frac{9}{7}\cdot\frac{1}{8^{n_0}}$, we get $0<\vert a_{n_0}- b_{n_0}\vert \leq 1$. We may assume that $b_{n_0}$ is larger than $a_{n_0}$ and thus obtain condition $(2)$. Hence  $\sum_{n=1}^{n_0}\frac{b_n}{8^n}-\sum_{n=1}^{n_0}\frac{a_n}{8^n}=\frac{1}{8^{n_0}}$. Since we do not have $8\in\Sigma$ and need to hold the inequality $\vert\sum_{n=1}^{n_0+1}\frac{b_n}{8^n}-\sum_{n=1}^{n_0+1}\frac{a_n}{8^n}\vert\leq\frac{1}{8^{n_0+1}}$, we have two possibilities. First of them appears if the equality $\sum_{n=1}^{n_0+1}\frac{b_n}{8^n}-\sum_{n=1}^{n_0+1}\frac{a_n}{8^n}=\frac{1}{8^{n_0+1}}$ holds. Then it means that $a_{n_0+1}-b_{n_0+1}=7$, which is described by $(5)$ condition. Otherwise if we have $\sum_{n=1}^{n_0+1}\frac{b_n}{8^n}-\sum_{n=1}^{n_0+1}\frac{a_n}{8^n}=-\frac{1}{8^{n_0+1}}$, then this case is described by the $(3)$ condition. We continue by induction.   
\end{proof}
\end{theorem}
In analogous way we prove the following general result.
\begin{corollary}\label{ogolne}
Assume that $x\in E(2^r,2^{r-1},\ldots,4,3,2;\frac{1}{2^{r+1}})$ has more than one digital representation. There exists the finite or infinite sequence of positive natural numbers $n_0<n_1<\ldots$ and exactly two digital representations $x=\sum_{n=1}^{\infty}\frac{a_n}{8^n}=\sum_{n=1}^{\infty}\frac{b_n}{8^n}$ of $x$ such that 
\begin{itemize}
\item $a_k=b_k$ for $k<n_0$;
\item $a_{n_0}=2+j$ and $b_{n_0}=3+j$ for some $j\in\{0,1,2,\ldots,2^{r+1}-4\}$;
\item $a_{n_k}=2^{r+1}+1$ and $b_{n_k}=0$ for odd $k$;
\item $a_{n_k}=0$ and $b_{n_k}=2^{r+1}+1$ for even $k>0$;
\item  $a_{i}\in\{2^{r+1}-1,2^{r+1}+1\}$ and $a_i-b_{i}=2^{r+1}-1$, as far as $n_{2k}<i<n_{2k+1}$;
\item  $a_{i}\in\{0,2\}$ and $b_i-a_{i}=2^{r+1}-1$, as far as $n_{2k+1}<i<n_{2k+2}$.
\end{itemize}
\end{corollary}
\begin{corollary}
By using similar argumentation to that in \cite{BPW} for any Cantorval $E$ constructed in Corollary \ref{ogolne}, it is clear that both $E_1$ and $E_2$ are dense. 
\end{corollary}

\section{Topological size of the uniqueness set}
In this section we will focus on topological size of the set $U(x_n)$. We show that it has the Baire property and we give sufficient condition for $U(x_n)$ to be comeager in $E(x_n)$ when $E(x_n)$ is a Cantorval. Some topological results of presented here can be proved in more general setting and they are very likely a mathematical folklore. We presented their proofs for the sake of completeness and for the reader convenience.  

We say that a function $f:\{0,1\}^\N\to\R$ is semi-open if $f(U)$ has non-empty interior in $\R$ for every non-empty open set $U$ in $\{0,1\}^\N$.
\begin{proposition}\label{PropositionSemiOpen}
Let $f:\{0,1\}^\N\to\R$ be a continuous semi-open function. Then 
 $f^{-1}(A)$ is nowhere dense in $\{0,1\}^\N$ for any nowhere dense subset $A$ of $\R$.
\end{proposition}
\begin{proof}
Let $U$ be a non-empty open subset of $\{0,1\}^\N$. Since $f$ is semi-open, there is a non-empty open set $V$ in $\R$ such that $V\subseteq f(U)$. The set $A$ is nowhere dense, so one can find a non-empty open $W\subseteq\R$ with $W\subseteq V\setminus A$. By the continuity of $f$, the pre-image $f^{-1}(W)$ is open as well. Moreover, $f^{-1}(W)\subset U\setminus f^{-1}(A)$, which shows that the set $f^{-1}(A)$ is nowhere dense in $\{0,1\}^\N$.  
\end{proof}

We have already noted that the function $(\varepsilon_n)\mapsto\sum_{n=1}^\infty\varepsilon_nx_n$ is a continuous mapping from $\{0,1\}^\N$ to $\R$. Now, we will show that it is also, in the case we are interested in here, a semi-open function. 
\begin{proposition}\label{fIsSemiOpen}
Let $(x_n)$ be a summable sequence of reals and let $f:\{0,1\}^\N\to\R$ be given by $f((\varepsilon_n)_{n=1}^\infty)=\sum_{n=1}^\infty\varepsilon_nx_n$. If $E(x_n)$ has a non-empty interior, then $f$ is semi-open. 
\end{proposition}
\begin{proof}
If $E(x_n)$ has a non-empty interior in $\R$, then $E(x_n)$ is either Cantorval or finite union of compact intervals. Let $U$ be a non-empty open subset of $\{0,1\}^\N$. Then $U$ contains a basic set, that is the set of the form $B_{(d_1,\dots,d_k)}=\{(\varepsilon_n)\in\{0,1\}^\N:\varepsilon_i=d_i$ for $i\leq k\}$ where $d_i=0,1$ are fixed 0-1 digits. Note that 
\[
f(B_{(d_1,\dots,d_k)})=\sum_{i=1}^kd_ix_i+E((x_n)_{n\geq k})
\]
By Proposition \ref{skonczonezmiany} the set $E((x_n)_{n\geq k})$ has a non-empty interior, so does $f(U)$.
\end{proof}

A subset $A$ of $\{0,1\}^\N$ is called Fin-invariant if 
\[
(\varepsilon_1,\varepsilon_2,\dots)\in A\iff (d_1,\dots,d_k,\varepsilon_{k+1},\varepsilon_{k+2},\dots)\in A
\]
for all $d_1,\dots,d_k\in\{0,1\}$. Note that a Fin-invariant set is invariant on changing finitely many coordinates. 
Recall that a set $A$ in a topological space has the Baire property if there are an open set $U$ and a meager set $M$ with $A=U\bigtriangleup M$, where $\bigtriangleup$ stands for the symmetric difference operator; in particular $A$ is comeager in $U$.  
\begin{proposition}\label{MeagerOrComeager}
Assume that $A$ has the Baire property and is Fin-invariant. Then $A$ is either meager or comeager subset of $\{0,1\}^\N$.
\end{proposition}

\begin{proof}
Assume that $A$ is not meager. Since $A$ has the Baire property, there is  basic open set $B_{(\bar{d}_1,\dots,\bar{d}_k)}$ such that $A$ is comeager in it. Note that
\[
\{0,1\}^\N=\bigcup_{(d_1,\dots,d_k)\in\{0,1\}^k}B_{(d_1,\dots,d_k)}.
\]
Thus by the Fin-invariance, $A$ is comeager in $\{0,1\}^\N$.  
\end{proof}

A subset $B$ of $E(x_n)$ is called Fin-invariant if 
\[
\sum_{i=1}^\infty\varepsilon_ix_i\in B\iff \sum_{i=1}^kd_ix_i+\sum_{i=k+1}^\infty\varepsilon_ix_i\in B
\]
for all $d_1,\dots,d_k\in\{0,1\}$.
\begin{proposition}
The set $U(x_n)$ is a co-analytic subset of $\R$. In particular it has the Baire property.
\end{proposition}
\begin{proof}
Let $f:\{0,1\}^\N\to\R$ be given by $f((\varepsilon_n)_{n=1}^\infty)=\sum_{n=1}^\infty\varepsilon_n x_n$. Since $f$ is continuous, its graph is closed subset of $\{0,1\}^\N\times\R$. Note that 
\[
U(x_n)=\{y\in\R:\exists!(\varepsilon_n)\;\;((\varepsilon_n),y)\in\text{graph}(f)\}
\]
where $\exists!$ stands for 'exists exactly one'. Therefore by Lusin Theorem \cite[Theorem 18.11]{kechris} the set $U(x_n)$ is co-analytic. By Lusin-Sierpi\'nski Theorem \cite[Theorem 21.6]{kechris} every analytic sets have the Baire property, which implies that co-analytic sets have the Baire property as well. 
\end{proof}

\begin{theorem}\label{UisMeagerOrComeager}
Assume that $E(x_n)$ has non-empty interior and $B\subseteq E(x_n)$ is Fin-invariant. Then  either $B$ is meager or $E(x_n)\setminus B$ is meager on $\R$.  
\end{theorem}

\begin{proof}
Let $f:\{0,1\}^\N\to\R$ be given by $f((\varepsilon_n)_{n=1}^\infty)=\sum_{n=1}^\infty\varepsilon_nx_n$. Let $A=f^{-1}(B)$ and $d_1,\dots,d_k\in\{0,1\}$. Then
\[
(\varepsilon_n)\in A\iff f(\varepsilon_n)\in B\iff\sum_{n=1}^\infty\varepsilon_nx_n\in B
\]
and by the Fin-invariance of $B$
\[
\sum_{n=1}^\infty\varepsilon_nx_n\in B\iff\sum_{i\leq k}d_ix_i+\sum_{i>k}\varepsilon_ix_i\in B\iff f(d_1,\dots,d_k,\varepsilon_{k+1},\varepsilon_{k+2},\dots)\in B\iff 
\]
\[\iff(d_1,\dots,d_k,\varepsilon_{k+1},\varepsilon_{k+2},\dots)\in A.
\]
Thus $A$ is Fin-invariant. Since $A$ is a continuous pre-image of a co-analytic set, it is co-analytic as well, and consequently it has the Baire property. By Proposition \ref{MeagerOrComeager} the set $A$ is either meager or comeager. 

Suppose that neither $B$ is meager nor $E(x_n)\setminus B$. Then there are two nonempty disjoint open sets $U,V\subset E(x_n)$ such that $B$ is comeager in $U$ and $E(x_n)\setminus B$ is comeager in $V$. By Proposition \ref{PropositionSemiOpen} and Proposition \ref{fIsSemiOpen} $f^{-1}(B)$ is meager in $f^{-1}(V)$ and $f^{-1}(E(x_n)\setminus B)$ is meager in $f^{-1}(U)$. This shows that $A$ is neither meager nor comeager, which yields a contradiction. 
\end{proof}
Note that the assertion of Theorem \ref{UisMeagerOrComeager} holds also without assumption that $E(x_n)$ has non-empty interior, but in this case it is trivial.

For a topological space $X$, by $\mathcal{M}_X$ we denote the family of all meager sets in $X$. By $\mathcal{M}_X\upharpoonright E$ we denote the of all restriction of meager sets to $E$, that is
\[
\mathcal{M}_X\upharpoonright E=\{A\subset X:A\subset E\text{ and }A\in\mathcal{M}_X\}.
\]
\begin{proposition}\label{MeagerInSubspaceEqualsMeagerRestricted}
Let $E\subseteq\R$ be a compact set. Then $\mathcal{M}_E=\mathcal{M}_\R\upharpoonright E$ if and only if
\begin{equation}\label{OpenContainsInteval}
\text{every open set in }E \text{ contains an open interval}. 
\end{equation}
\end{proposition}

\begin{proof}
Every meager subset of subspace $E$ is also a meager in the whole space $\R$. Thus $\mathcal{M}_E\subseteq\mathcal{M}_\R\upharpoonright E$ holds no matter what $E$ is. We need to show that the condition 
\begin{equation}\label{NWDisNWD}
\text{every nowhere dense subset of }\R\text{ contained in }E\text{ is nowhere dense in }E    
\end{equation}
is equivalent to condition \eqref{OpenContainsInteval}. 

Assume $\neg\eqref{OpenContainsInteval}$. Then there is a set $V$ open in $E$ which does not contain open interval. Since $E$ is closed, then $\text{cl}_{\R}(V)\subset E$. Since $\text{cl}_{\R}(V)=\text{cl}_{E}(V)$, then $\text{Int}_E(\text{cl}_{\R}(V))=V$, which implies that $\text{cl}_{\R}(V)$ does not contain any open interval, and consequently it is nowhere dense in $\R$. So is $V$ and we obtain $\neg\eqref{NWDisNWD}$.

Now, let us assume $\neg\eqref{NWDisNWD}$. Then there is a set $A$ nowhere dense in $\R$ which contains a non-empty set $V$ open in $E$. Since $V$ is nowhere dense in $\R$, it does not contain any open interval, and we obtain $\neg\eqref{OpenContainsInteval}$.  
\end{proof}

\begin{corollary}\label{CorollaryMeagerMeager}
Assume that $E=E(x_n)$ contains an open interval. Then $\mathcal{M}_E=\mathcal{M}_\R\upharpoonright E$. 
\end{corollary}

\begin{proof}
By Proposition \ref{MeagerInSubspaceEqualsMeagerRestricted} it is enough to show that any open set in $E$ contains an interval. Since $E$ is the achievement set containing open interval, it is either a finite union of compact intervals or a Cantorval. If $E$ is a finite union of compact intervals, the assertion is clear. If $E$ is Cantorval, then any its point $x$ is a limit of $E$-intervals, that is any neighbourhood of $x$ contains a connected component of $E$ which is an interval \cite{BFPW}. 
\end{proof}

Using Corollary \ref{CorollaryMeagerMeager} we can restate Theorem \ref{UisMeagerOrComeager} as follows

\begin{corollary}\label{MeagerComeagerCorollary}
Assume that $E(x_n)$ is a Cantorval or a finite union of compact intervals and $B\subseteq E(x_n)$ is Fin-invariant. Then $B$ is either meager or comeager in $E(x_n)$.  \end{corollary}

\begin{proposition}
Let $f:\{0,1\}^\N\to\R$ be a continuous semi-open function. Assume that the set $B:=\{x\in\{0,1\}^\N: \{x\}=f^{-1}(f(x))\}$ is nowhere dense in $\{0,1\}^\N$. Then 
$f(B)$ is nowhere dense in $\R$.
\end{proposition}

\begin{proof}
Let $U$ be a non-empty subset of $\R$. Since $B$ is nowhere dense in $\{0,1\}^\N$, there is non-empty open subset $V$ of open set $f^{-1}(U)$ with $V\cap B=\emptyset$. Since $f$ is semi-open, there is a non-empty open set $W$ contained in $f(V)$. Note that $f(V)\cap f(B)=\emptyset$; otherwise there would be $x\in B$ and $z\in U$ with $f(x)=f(z)$, and $\{x\}\neq\{x,z\}\subseteq f^{-1}(f(x))$ which yields a contradiction with the definition of $B$. Since $W\subset U\setminus f(B)$, then $f(B)$ is nowhere dense. 
\end{proof}

\begin{corollary}\label{CorNWDUniqueSet}
Assume that $E(x_n)$ has a non-empty interior. Let $f((\varepsilon_n)_{n=1}^\infty)=\sum_{n=1}^\infty\varepsilon_n x_n$. If $f^{-1}(U(x_n))$ is nowhere dense, so is $U(x_n)$.
\end{corollary}

Let $(k_0,k_1,\dots,k_m;q)$ be a multigeometric series. By $\Sigma$ we denote all subsums of $\{k_0,\dots,k_m\}$, that is 
\[
\Sigma=\{\sum_{i=0}^m\varepsilon_ik_i:\varepsilon_i=0,1\}.
\]
Then $E(k_0,k_1,\dots,k_m;q)=\{\sum_{n=0}^\infty a_nq^n: (a_n)\in\Sigma^\N\}$. By $U(\Sigma)$ we denote those points in $\Sigma$ which have unique representations. 

\begin{proposition}\label{PropSetUniqIsNWD}
Assume that $E(k_0,k_1,\dots,k_m;q)$ has a non-empty interior. If $U(\Sigma)\neq\Sigma$, then the set of uniqueness $U(k_0,k_1,\dots,k_m;q)$ is nowhere dense.  
\end{proposition}

\begin{proof}
Let $(x_n)=(k_0,k_1,\dots,k_m;q)$ and $f((\varepsilon_n)_{n=1}^\infty)=\sum_{n=1}^\infty\varepsilon_nx_n$.
Let $\sigma\in\Sigma\setminus U(\Sigma)$. Then there are two distinct tuples $(\varepsilon_0',\dots,\varepsilon_m')$ and $(\varepsilon_0'',\dots,\varepsilon_m'')$ in $\{0,1\}^{m+1}$ such that $\sigma=\sum_{i=0}^m\varepsilon'_ik_i=\sum_{i=0}^m\varepsilon''_ik_i$. Consider a set $X$ given by
\[
\{(\varepsilon_i)\in\{0,1\}^\N:(\varepsilon_{p(m+1)},\varepsilon_{p(m+1)+1},\dots,\varepsilon_{p(m+1)+m})\neq (\varepsilon'_{p(m+1)},\varepsilon'_{p(m+1)+1},\dots,\varepsilon'_{p(m+1)+m})\text{ for every }p\in\N\}.
\]
Note that $X$ is closed with an empty interior. Thus $X$ is nowhere dense. Note also that $f^{-1}(U(x_n))\subseteq X$. Thus by Corollary \ref{CorNWDUniqueSet} the set of uniqueness is nowhere dense.
\end{proof}

Using a similar argument one can prove the following extension of Proposition \ref{PropSetUniqIsNWD}.
\begin{theorem}\label{MultigeometricWithNowhereDenseUniqueSet}
Assume that $E(k_0,k_1,\dots,k_m;q)$ has a non-empty interior. If $U(\Sigma+\Sigma q+\dots\Sigma q^k)\neq\Sigma+\Sigma q+\dots\Sigma q^k$, then $U(k_0,k_1,\dots,k_m;q)$ is nowhere dense.
\end{theorem}
  
Now let us present a similar fact in more general settings. 

\begin{theorem}
Let $E(x_n)$ have a non-empty interior. Assume that there exists a sequence $(I_k^i)_{k\in\N, i\in\{0,1\}}$ of pairwise disjoint subsets of $\N$ such that 
\begin{itemize}
    \item $I^0_k$ is finite for every $k$;
    \item $\sum_{n\in I^0_k}x_n=\sum_{n\in I^1_k}x_n$ for every $k$.
\end{itemize}
Then $U(x_n)$ is nowhere dense in $E(x_n)$.
\end{theorem}
\begin{proof}
Let $f((\varepsilon_n)_{n=1}^\infty)=\sum_{n=1}^\infty\varepsilon_nx_n$. Consider the following set
\[
X:=\{(\varepsilon_n)\in\{0,1\}^\N:\forall k\exists n\in I^0_k\;\;(\varepsilon_n\neq 1)\}.
\]
Note that $U_k:=\{(\varepsilon_n)\in\{0,1\}^\N:\exists n\in I^0_k\;\;(\varepsilon_n\neq 1)\}$ is clopen subset of $\{0,1\}^\N$. Since $X$ is an intersection of all $U_k$'s, $X$ is closed. Moreover $X$ has an empty interior, and therefore it is nowhere dense. Finally note that $f^{-1}(U(x_n))\subseteq X$. The assertion follows from Corollary \ref{CorNWDUniqueSet}. 
\end{proof}

Now, let us consider the uniqueness set of Guthrie-Nymann's Cantorval. 

\begin{theorem}\label{GNCharacterizationUnique}
Let $(a_n)\in\{0,2,3,5\}^\N$. The following conditions are equivalent 
\begin{itemize}
    \item[(1)] there is $(b_n)\in\{0,2,3,5\}^\N$ such that $(a_n)\neq (b_n)$ and
\[
\sum_{n=1}^\infty\frac{a_n}{4^n}=\sum_{n=1}^\infty\frac{b_n}{4^n}.
\]
    \item[(2)] $\Big\{n\in\N:(a_n,a_{n+1})\in\{(2,3),(3,2), (3,0),(2,5)\}\Big\}$ is finite non-empty.
\end{itemize}
\begin{proof}
"$\Rightarrow$". By \cite{BPW} there exists $n_0<n_1<n_2<\ldots$ such that
\begin{itemize}
    \item $n_0=\min\{n: a_n\neq b_n\}$ and $a_{n_0}=2$, $b_{n_0}=3$ (either vice-versa and the notations of $(a_n)$ and $(b_n)$ change with each other in the next points) 
    \item $a_{n_k}=5$, $b_{n_{k}}=0$ for odd $k$
    \item $a_{n_k}=0$, $b_{n_{k}}=5$ for even $k>0$
    \item $a_i\in\{3,5\}$, $a_i-b_i=3$ for every $n_{2k}<i<n_{2k+1}$
     \item $a_i\in\{0,2\}$, $b_i-a_i=3$ for each $n_{2k+1}<i<n_{2k+2}$
\end{itemize}
\begin{center}
$(a_n)$  \ \ $2$ $5 \ 5\ \ldots \ 5 \atop 3\ 3\ \ldots \ 3$ $5$ $2 \ 2\ \ldots \ 2 \atop 0\ 0\ \ldots \ 0$ $0$ $5 \ 5\ \ldots \ 5 \atop 3\ 3\ \ldots \ 3$ $5$ $\ldots$
\end{center}
\begin{center}
$(b_n)$  \ \ $3$ $2 \ 2\ \ldots \ 2 \atop 0\ 0\ \ldots \ 0$ $0$ $5 \ 5\ \ldots \ 5 \atop 3\ 3\ \ldots \ 3$ $5$ $2 \ 2\ \ldots \ 2 \atop 0\ 0\ \ldots \ 0$ $0$ $\ldots$
\end{center}
Suppose that $a_m=2$ for some $m>n_0$. Then there exists $k$ such that $n_{2k+1}<m<n_{2k+2}$. If $m+1<n_{2k+2}$, then $a_{m+1}=0$ either $a_{m+1}=2$. Otherwise $m+1=n_{2k+2}$ and we have $a_{m+1}=0$. It means that $\Big\{n\in\N:(a_n,a_{n+1})\in\{(2,3),(2,5)\}\Big\}\cap (n_0,\infty)=\emptyset$. In a very similar way we obtain that $\Big\{n\in\N:(a_n,a_{n+1})\in\{(3,2),(3,0)\}\Big\}\cap (n_0,\infty)=\emptyset$. Hence we have already shown that the set  $\Big\{n\in\N:(a_n,a_{n+1})\in\{(2,3),(3,2), (3,0),(2,5)\}\Big\}$ is finite. Now we prove that it is non-empty.

If $a_{n_0}=2$ we have two possibilities. If $n_0+1<n_1$, then $a_{n_0+1}=3$ either $a_{n_0+1}=5$. On the other hand, when $n_0+1=n_1$, we get $a_{n_0+1}=5$. Thus $(a_{n_0},a_{n_0+1})\in\{(2,3),(2,5)\}$. 

Now, assume that $a_{n_0}=3$. Then in a similar way we obtain that $a_{n_0+1}$ is equal to $0$ either $2$. Finally we get that $\Big\{n\in\N:(a_n,a_{n+1})\in\{(2,3),(3,2), (3,0),(2,5)\}\Big\}$ contains $n_0$, that is the first index which differs the sequences $(a_n)$ and $(b_n)$.
\\"$\Leftarrow$". Let $n_0=\max\{n\in\mathbb{N} : (a_{n},a_{n+1})\in \{(2,3),(3,2),(3,0),(2,5)\}\}$. Let $B\subset\mathbb{N}$ be defined as $B=\{n>n_0 : (a_{n},a_{n+1})\in \{(5,0),(5,2),(0,3),(0,5)\}\}=\{n_1<n_2<\ldots\}$, which may be finite or infinite. Without loss of generality let $a_{n_1}=5$. Then $a_{n_1+1}=0$ either $a_{n_1+1}=2$. By assumption if $a_k=2$ for some $k>n_1$, then $a_{k+1}=0$ either $a_{k+1}=2$ and thus $k\notin B$. If $a_k=0$ and $a_{k+1}=0$, then also $k\notin B$, while $a_k=0$ and $a_{k+1}=3$ either $a_{k+1}=5$ implies that $k\in B$. Hence $a_{n_2}=0$. In the same way we prove that $a_{n_{2k-1}}=5$ and  $a_{n_{2k}}=0$ for each $k\in\mathbb{N}$. Furthermore $a_{i}\in\{0,2\}$ for every $n_{2k+1}<i<n_{2k+2}$ and $a_{i}\in\{3,5\}$ for every $n_{2k}<i<n_{2k+1}$. Define $b_i=a_i$ for each $i\leq n_0$, $b_{n_{2k-1}}=0$ and $b_{n_{2k}}=5$ for every $k\in\mathbb{N}$. Moreover let $a_i-b_i=3$ for every $n_{2k}<i<n_{2k+1}$ and $b_i-a_i=3$ for each $n_{2k+1}<i<n_{2k+2}$. Thus we have $\sum_{n=1}^{\infty}\frac{a_n}{4^n}=\sum_{n=1}^{\infty}\frac{b_n}{4^n}$.

\end{proof}
\end{theorem}

The values in the following diagram are the possibilities for $a_n$ for $n>n_0$. The arrows shows what number can be the next term $a_{n+1}$. There are lack of four arrows describing $(2,3),(3,2),(3,0),(2,5)$ as the possible values for $(a_n,a_{n+1})$. The $(n_k)$ indexes are the moments when we move from the left part of the diagram to the right and vice-versa.
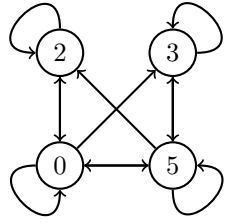
\begin{center}
\begin{tikzpicture}[node distance={15mm}, thick, main/.style = {draw, circle}] 
\node[main] (1) {$2$}; 
\node[main] (2) [below of=1] {$0$}; 
\node[main] (3) [right of=2] {$5$}; 
\node[main] (4) [above of=3] {$3$}; 
\draw[->] (1) -- (2); 
\draw[->] (2) -- (1); 
\draw[->] (2) -- (3); 
\draw[->] (2) -- (4); 
\draw[->] (3) -- (1);
\draw[->] (3) -- (2);
\draw[->] (3) -- (4);
\draw[->] (4) -- (3);
\draw[->] (1) to [out=90,in=180,looseness=5] (1); 
\draw[->] (2) to [out=180,in=270,looseness=5] (2); 
\draw[->] (3) to [out=270,in=0,looseness=5] (3); 
\draw[->] (4) to [out=0,in=90,looseness=5] (4); 
\end{tikzpicture} 
\end{center}

\begin{theorem}\label{SetOfUniqGutrieNymann}
$U(3,2;\frac{1}{4})$ is comeager in $E(3,2;\frac{1}{4})$.
\end{theorem}

\begin{proof}
Recall that the elements of $E(3,2;\frac{1}{4})$ can be represented either as $\sum_{n=1}^\infty\varepsilon_nx_n$, where $\varepsilon_n=0,1$, $x_{2n-1}=\frac{2}{4^{n-1}}$ and $x_{2n}=\frac{3}{4^{n-1}}$, or as $\sum_{n=1}^\infty \frac{a_n}{4^n}$, where $a_n=0,2,3,5$. Note that the set 
\[
B_{(2,3)}=\{(\varepsilon_n)\in\{0,1\}^\N:\varepsilon_{2n-1}=0,\varepsilon_{2n}=1, \varepsilon_{2n+1}=1, \varepsilon_{2n+2}=0\text{ for infinitely many $n$'s}\}
\]
represents the same elements as the set of all sequences $(a_n)$ with $(a_n,a_{n+1})=(2,3)$ for infinitely many $n$'s. Using this observation we define the following sets
\[
B_{(3,2)}=\{(\varepsilon_n)\in\{0,1\}^\N:\varepsilon_{2n-1}=1,\varepsilon_{2n}=0, \varepsilon_{2n+1}=0, \varepsilon_{2n+2}=1\text{ for infinitely many $n$'s}\}
\]
\[
B_{(3,0)}=\{(\varepsilon_n)\in\{0,1\}^\N:\varepsilon_{2n-1}=1,\varepsilon_{2n}=0, \varepsilon_{2n+1}=0, \varepsilon_{2n+2}=0\text{ for infinitely many $n$'s}\}
\]
\[
B_{(2,5)}=\{(\varepsilon_n)\in\{0,1\}^\N:\varepsilon_{2n-1}=0,\varepsilon_{2n}=1, \varepsilon_{2n+1}=1, \varepsilon_{2n+2}=1\text{ for infinitely many $n$'s}\}.
\]
By $B$ we denote the union of those four sets $B_{(2,3)}, B_{(3,2)}, B_{(3,0)}$ and $B_{(2,5)}$. Note that each of those sets is comeager and Fin-invariant, and so is $B$. Note that $f(B)$ is Fin-invariant as well, where $f((\varepsilon_n)_{n=1}^\infty)=\sum_{n=1}^\infty\varepsilon_nx_n$ as usual. By Theorem \ref{GNCharacterizationUnique} we obtain that $f(B)\subset U(3,2;\frac{1}{4})$. Since $f(B)$ is analytic, it has the Baire Property. By Corollary \ref{MeagerComeagerCorollary} the set $f(B)$ is either meager or comeager. Since $B\subset f^{-1}(f(B))$, then by Proposition \ref{PropositionSemiOpen} and Proposition \ref{fIsSemiOpen}, $f(B)$ is not meager. Thus $U$ as a superset of comeager set $f(B)$ is comeager as well. 
\end{proof}

The Guthrie-Nymann's Cantorval is an achievement set given by $E(2,3;\frac{1}{4})$. So $\Sigma=\{2,3\}$ and $q=\frac{1}{4}$. By Theorem \ref{MultigeometricWithNowhereDenseUniqueSet} and Theorem \ref{SetOfUniqGutrieNymann} we obtain that $U(\Sigma+\Sigma q+\dots\Sigma q^k)=\Sigma+\Sigma q+\dots\Sigma q^k$ for every $k\in\N$.

\section{Cantorvals with  continuum points}
It is not hard to construct $E(x_n)$ with point which has $\mathfrak{c}$ - many expansions. One method was described in \cite{cardfun}: we start with any interval filling sequence $(y_n)$ and simply repeat two times any of its terms, that is $x_{2n-1}=x_{2n}=y_n$.
It is both a universal method of transforming any interval filling sequence into a locker and a construction of $E(x_n)$, which is an interval with all interiors points obtained for $\mathfrak{c}$ - many representations. 

There exist many Cantorvals with $\mathfrak{c}$-points. To see it let us consider a multigeometric series $E(k_1,\ldots,k_m; q)$ such that there exists $x\in\Sigma$ with at least two representations $x=\sum_{i\in A}k_i=\sum_{i\in B}k_i$ for $A,B\subset \{1,\ldots,m\}$, $A\neq B$. Thus $y=\frac{x}{1-q}=\sum_{n\in C}x_n$, where $C$ can be any element of the family $\{\bigcup_{n\in\mathbb{N}_0}C_n :C_n=A+m\cdot n \ \text{or} \ C_n=B+m\cdot n \ \text{for all} \ n\in\mathbb{N}_0\}$, i.e. $y$ is a $\mathfrak{c}$ - point. Moreover, note that we can find more $\mathfrak{c}$ - points: if  $z=\sum_{n\in D}x_n$, for some $D\in\mathcal{D}=\{\bigcup_{n\in\mathbb{N}_0}D_n :D_n=A+m\cdot n \ \text{or} \ D_n=B+m\cdot n \ \text{or} \ D_n=\emptyset \ \text{for all} \ n\in\mathbb{N}_0 \ \text{and} \ \vert\{n\in\mathbb{N}_0 : D_n\neq\emptyset\}\vert=\infty\}$, then it has $\mathfrak{c}$ - many more expansions in $\mathcal{D}$, so there are $\mathfrak{c}$ many $\mathfrak{c}$ - points. Note that in \cite{cardfun} the authors proved that if there exists $\mathfrak{c}$ - point, then the set of all such points is dense in $E(x_n)$.  
\begin{example}
Let us consider a Guthrie-Nymann-Jones sequence $E(3,2,2,2;q)$. By \cite{BBFS} we know that it is a Cantorval for any $q\in(\frac{1}{6},\frac{2}{11})$. Then $\Sigma=\{0,2,3,4,5,6,7,9\}$ and $2$ has three representations in $\Sigma$. Hence $\frac{2}{1-q}$ is $\mathfrak{c}$ - point.
\end{example}
The problem can be considered in another way. For an achievement set $E(y_n)$ to construct its superset $E(x_n)$ in the way $(x_n)\supset (y_n)$ and $E(x_n)$ contains an $\mathfrak{c}$ - point. Moreover we want $E(x_n)$ to be homeomorphic to $E(y_n)$. If $E(y_n)$ is an interval we do not have to worry about its interval-structure. The problem appears when $E(y_n)$ is a Cantroval. Namely when we add too much terms, we may lose the gaps and obtain an interval filling sequence $(x_n)$. It is illustrated in the following Example. 
\begin{example}
Let $E(y_n)$ be a Guthrie-Nymann's Cantorval, that is $y_{2n-1}=\frac{3}{4^n}$, $y_{2n}=\frac{2}{4^n}$ for $n\in\mathbb{N}$. Let us consider the sequence $x_{2n-1}=x_{2n}=y_n$ for each $n\in\mathbb{N}$. Thus $(x_n)$ is interval filling, so $E(x_n)=[0,\frac{10}{3}]$.  
\end{example}
Let us now improve the above construction. 
\begin{example}\label{powiekszonyGN}
Let $E(y_n)$ be a Guthrie-Nymann's Cantorval, that is $y_{2n-1}=\frac{3}{4^n}$, $y_{2n}=\frac{2}{4^n}$ for $n\in\mathbb{N}$. Let us consider the sequence $x_{5n-4}=y_{4n-3}$, $x_{5n-3}=x_{5n-2}=y_{4n-2}$, $x_{5n-1}=y_{4n-1}$, $x_{5n}=y_{4n}$ for each $n\in\mathbb{N}$, that is we repeats two times the elements $y_{4n-2}$ for all $n$. Note that $$x_{5n-2}=y_{4n-2}=\frac{2}{4\cdot 16^{n-1}}>\frac{5}{12\cdot 16^{n-1}}+\frac{1}{30\cdot 16^{n-1}}=\sum_{k>4n-2}y_k+\sum_{k=n}^{\infty}y_{4k+2}=\sum_{k>5n-2}x_k$$ 
Hence $(\sum_{k>5m-2}x_k,x_{5m-2})$ is a gap in $E(x_n)$ for every $m\in\mathbb{N}$. By inclusion $E(x_n)\supset E(y_n)$, we obtain that $E(x_n)$ is a Cantorval. Finally, note that the sum of any infinite subseries of the series $\sum_{k=1}^{\infty}x_{5k-2}$ is a $\mathfrak{c}$ - point. In fact, there are $\mathfrak{c}$ many $\mathfrak{c}$ - points.
\end{example}
Now let us apply the idea staying behind Example \ref{powiekszonyGN} for any Cantorval.
\begin{theorem}\label{wiekszycantorval}
Let $E(y_n)$ be a Cantorval. Then there exists a Cantorval $E(x_n)\supset E(y_n)$ which has $\mathfrak{c}$ - point.
\begin{proof}
Let $(a_n,b_n)$ be the sequence of the longest gaps from the left in $E(y_n)$, where $(b_n)$ is decreasing. Note that by the Third Gap Lemma $(b_n)\subset (y_n)$. Let $n_1$ be a natural number such that $\sum_{n=n_1}b_n<b_1-a_1$.  We continue the construction by induction and obtain a sequence $(n_k)$ such that $\sum_{n=n_{k+1}}^{\infty}b_n<b_{n_k}-a_{n_k}$ for every $k\in\mathbb{N}$. Let $(x_n)=(y_n)\cup (b_{n_k})$. Hence $E(x_n)$ has still the gaps with the end $b_{n_k}$ for each $k\in\mathbb{N}$, so it is a Cantorval. Note that for each $k\in\mathbb{N}$ the term $b_{n_k}$ appears at least twice in the sequence $(x_n)$. Hence we obtain that sum of any infinite subseries of $\sum_{k=1}^{\infty}b_{n_k}$ is $\mathfrak{c}$ - point.
\end{proof}
\end{theorem}

On the other hand there exists Cantorval $E(y_n)$ with $\mathfrak{c}$ points with the property that for every $(x_n)\subset (y_n)$ such that $E(x_n)$ has no $\mathfrak{c}$-points, then $E(x_n)$ is not a Cantorval.

\begin{example}
Let $E(y_n)=E(3,2,2;\frac{1}{6})$. Let $(x_n)\subset (y_n)$ be such that $E(x_n)$ has no $\mathfrak{c}$ points. Then almost all of the terms $(\frac{2}{6^n})$ are in $(y_n)\setminus (x_n)$. Since $E(3,2;\frac{1}{6})$ is a Cantor set, by Proposition \ref{skonczonezmiany} we obtain that $E(x_n)$ is also a Cantor set if $(x_n)$ is infinite either it is a finite set.
\end{example}

\section{Minimal representation and unique points}
As we have proved  in Theorem \ref{wiekszycantorval} by adding terms we may generate $\mathfrak{c}$ - points without changing Cantorval structure. On the other hand now we show it strongly decreases the set $U(x_n)$. 

\begin{theorem}
Let $E(y_n)$ be a Cantorval or a finite sum of closed intervals and $B$ be the sum of the set of all left edges of intervals contained in $E(y_n)$ and its trivial components. Let $(b_n)$ be any sequence tending to $0$. Then $U(x_n)\cap E(y_n)\subset B$, where $(x_n)=(y_n)\cup (b_n)$.  Moreover each point $x\in E(y_n)\setminus B$ has $\mathfrak{c}$ - many expansions.  
\begin{proof}
Fix $x\in E(y_n)\setminus B$. Then there exists $b_k$ such that $x-b_k\in E(y_n)$. Since $x=x-b_k+b_k$, we get another expansion of $x$. To prove the 'moreover' part it is enough to observe that instead of a single element $b_k$ above, we may take the tails of the quickly convergent subseries $\sum_{k=1}^{\infty} b_{n_k}$  of the series $\sum_{n=1}^{\infty} b_{n}$. 
\end{proof}
\end{theorem}
In particular $U(x_n)$ is not dense in $E(x_n)$. The above result inspires to introduce the following property of sequence.
\begin{definition}
A sequence $(x_n)$ is a minimal representation of a Cantorval (or finite sum of closed intervals) $E(x_n)$ iff after removing any its infinite subsequence the achievement set $E(y_n)$  for the remaining sequence $(y_n)\subset (x_n)$ is not a Cantorval (a finite sum of closed intervals, respectively). 
\end{definition}
\begin{corollary}
If $U(x_n)$ is dense in $E(x_n)$, then $(x_n)$ is minimal representation.
\end{corollary}
\begin{corollary}
If $E(x_n)$ has no $\mathfrak{c}$-points, then $(x_n)$ is minimal representation.
\end{corollary}
\begin{example}
Any sequence $(2^r,2^{r-1},\ldots,4,3,2;\frac{1}{2^{r+1}})$ for $r\in\mathbb{N}$ is a minimal representation of a Cantorval. In particular if $r=1$, then the sequence $(3,2;\frac{1}{4})$ is a minimal representation of  the Guthrie and Nymann's Cantorval.
\end{example}
\begin{example}
Let us consider $E(3,\underbrace{2,2,\ldots,2}_{m - \ \text{times}};q)$. By \cite{BFS} we know that $E$ is a Cantorval for all $q\in[\frac{1}{2m},\frac{2}{2m+5}]$. After removing subsequence $(\frac{2}{q^n})$, we obtain $E(3,\underbrace{2,2,\ldots,2}_{m-1 - \ \text{times}};q)$, which is a Cantorval for $q\in [\frac{1}{2m-2},\frac{2}{2m+3}]$. Note that for $m\geq 5$ we have $[\frac{1}{2m},\frac{2}{2m+5}]\cap [\frac{1}{2m-2},\frac{2}{2m+3}]=[\frac{1}{2m-2},\frac{2}{2m+5}]\neq\emptyset$. Hence for $m\geq 5$ and $q\in [\frac{1}{2m-2},\frac{2}{2m+5}]$ a sequence $(3,\underbrace{2,2,\ldots,2}_{m - \ \text{times}};q)$ is not a minimal representation of Cantorval $E$. In particular $U$ is not dense in $E$. 
\end{example}
The notion of minimal representation can be also applied for $E$, which is not a Cantorval.
\begin{example}\label{geom}
Let us consider a geometric sequence $(q^n)$. In the papers \cite{DJK} and \cite{Erdos} the ratio $q_0=\frac{\sqrt{5}-1}{2}$ is on the special interest. Basically for any $q\in[q_0,1)$ the geometric sequence is a locker, which suffices for $U=\{0,\frac{q}{1-q}\}$, while for any $q\in[\frac{1}{2},q_0)$ the set $U$ has more than two elements.
\begin{enumerate}
\item For $q\in[\frac{1}{2},q_0]$ it is a minimal representation of the interval $[0,\frac{q}{1-q}]$;
\\Indeed if $q\in[\frac{1}{2},q_0)$, then by removing any element $q^k$ we form a gap in $E$ which right edge is the smallest not removed element larger than $q^k$. Hence we obtain infinite many gaps, so $E$ is a Cantorval or even a Cantor set.
\\For $q=q_0$ the locker condition is satisfied with equality, so by removing any two elements $q^k$ and $q^m$ we again form a gap in $E$. 
\item For $q\in(q_0,1)$ it is not a minimal representation of the interval $[0,\frac{q}{1-q}]$;
\\Let us remove $q^{k_1}$ from our series. Then $\delta =r_{k_1}-q^{k_1-1}>0$. One can find $k_2$ such that $r_{k_2-1}<\delta$, that is removing elements with indexes larger or equal $k_2$ from the tail $r_{k_1}$ will not break the previous inequality. Let us remove $q^{k_2}$ from our series. We continue by induction and remove the terms $(q^{k_n})$ from the geometric sequence. The achievement set for the remaining sequence is still an interval.  
\end{enumerate}

\end{example}

Let us compare the both notions of locker and minimal sequence by characterizing them in Kakeya-like conditions.
Note that a sequence $(x_n)$ is a locker iff $x_k\leq r_{k+1}$ for all $k\in\mathbb{N}$. We will prove the following.

\begin{theorem}\label{Kakeyanormal}
Let $(x_n)$ be such that $E(x_n)$ is a finite sum of closed intervals. Then $(x_n)$ is not minimal iff $x_k<r_{k+1}$ for infinitely many $k$'s.
\begin{proof}
$\Rightarrow.$ Let us assume that $(x_n)$ is not minimal. It means that there exists an increasing sequence of natural numbers $(n_i)$ such that for the subsequence $(x_{1},x_{2},\ldots,x_{n_1},x_{n_2+1},x_{n_2+2},\ldots,x_{n_3},x_{n_4+1},\ldots)=(y_n)$ we have $E(y_n)$ is a finite sum of closed intervals.  That is, we have removed from $(x_n)$ consecutive terms with indexes $n_i+1,n_i+2,\ldots,n_{i+1}$ for each odd $i\in\mathbb{N}$. Hence for every large enough (to satisfy Kakeya condition for achievement set to be a sum of intervals) odd $i$ we have:
$$x_{n_i}\leq \sum_{j\in2\mathbb{N},j>i}\sum_ {k=n_j+1}^{n_{j+1}}x_k <\sum_{k\geq n_{i+1}}x_k\leq \sum_{k\geq n_{i}+2}x_k=r_{n_i+1}$$
\\$\Leftarrow.$  Assume that there exists an increasing sequence of natural numbers $(n_i)$ such that $r_{n_i+1}>x_{n_i}$ for every $i\in\mathbb{N}$. Without losing generality we may assume that for $n\geq n_1$ the sequence $(x_n)$ is slowly convergent. 
In the beginning we remove the term $x_{n_1+1}$. Let $k_1$ be such that $r_{n_{k_1}}<r_{n_1+1}-x_{n_1}$. We remove the term $x_{n_{k_1}+1}$. On can find $k_2$ satisfies the inequality $\min\{r_{n_{1}+1}-x_{n_1}-x_{n_{k_1}+1},r_{n_{k_1}+1}-x_{n_{k_1}}\}>r_{n_{k_2}}$. We omit the term $x_{n_{k_2}+1}$. Now let $k_3$ be such that 
$$\min\{r_{n_{1}+1}-x_{n_1}-x_{n_{k_1}+1}-x_{n_{k_2}+1},r_{n_{k_1}+1}-x_{n_{k_1}}-x_{n_{k_2}+1},r_{n_{k_2}+1}-x_{n_{k_2}}\}>r_{n_{k_3}} $$
We continue by induction and obtain a subsequence $(x_{n_{k_i}})$ of removed terms such that the achievement set for the remaining sequence is a finite sum of closed intervals.
\end{proof}
\end{theorem}

\begin{corollary}\label{niegesty}
Let $(x_n)$ be such that $E(x_n)$ is a finite sum of closed intervals. Assume that $x_k<r_{k+1}$ for infinitely many $k$'s, then $U(x_n)$ is not dense in $E(x_n)$ and $E(x_n)$ contains $\mathfrak{c}$ - points.
\end{corollary}

\begin{corollary}\label{twoinfgaps}
From Theorem \ref{Kakeyanormal} we know that a sequence $(x_n)$ which achievement set $E(x_n)$ is a finite sum of closed intervals is minimal iff $x_k\geq r_{k+1}$ for almost all $k$'s.
Both of the equivalent conditions have nice geometric interpretation.
Note that for interval-like achievement sets its minimality means: removing any infinite subsequence makes infinite many gaps, 
while the condition $x_k\geq r_{k+1}$ for all but finitely many $k$'s can be read as follows: removing any two far enough terms makes a new gap in the remaining set of subsums. 

Note also that the minimality for Cantorval-case has a different sense, namely after deletion of any infinite subsequence the left achievement set loses all its intervals. 

\end{corollary}

In Example \ref{geom} we have shown that typical sequence is a non-minimal locker either minimal but not locker. 
Although the both conditions being a locker and not minimal seem to be close to each other, none of them implies another. 
\begin{example}\label{zlotyciag}
The geometric sequence $(x_n)=(q_0^n)$ for $q_0=\frac{\sqrt{5}-1}{2}$ is a minimal locker, which is also a counterexample for reversal the implication in Corollary \ref{niegesty}. Indeed $U(x_n)=\{0,\sum_{n=1}^{\infty}x_n\}$ is far from being dense in the interval $E(x_n)$. Moreover $x=\sum_{n=1}^{\infty}x_{3n-2}$ is its $\mathfrak{c}$-point, since $x=\sum_{n\in\bigcup_{i\in\mathbb{N}} A_i}x_{n}$ for all sequences of sets of indices $(A_i)$ such that $A_i=\{3i-2\}$ either $A_i=\{3i-1,3i\}$ for each $i\in\mathbb{N}$. Also note that for every $k\in\mathbb{N}$ the equality $x_k= r_{k+1}$ holds. 
\end{example}
 In the next Example we construct the last possibility, that is the  sequence, which is not minimal nor locker. 

\begin{example}\label{lastpossible}
Let $(x_n)=(7,2,2,1,1,\frac{1}{2},\frac{1}{2},\frac{1}{4},\frac{1}{4},\frac{1}{8},\frac{1}{8},\ldots)$. 
\\Then it is not locker, since $x_1=7>6=r_2$. On the other hand $E(x_n)=[0,15]$ and $E(x_{2n})=[0,4]$ are both intervals, so $(x_n)$ is not a minimal representation. 
\end{example}

Example \ref{lastpossible} shows that a locker condition is sensitive to adding single term, while being minimal is not affected by adding or removing finitely many terms.

Now let us consider the following two conditions: having a $\mathfrak{c}$ - point and $U(x_n)$ being not dense in $E(x_n)$. 
Both of them are implied by the non-minimality of the sequence $(x_n)$. However it does not mean that one of them implies another or even that they are equivalent, although it may be suggested by the previously constructed Examples.






In the end of this section we give the Example of sequence for which neither $f^{-1}(1)$ nor $f^{-1}(\mathfrak{c})$ is dense in its achievement set. 

\begin{example}\label{jednopowtorzenie}
Let $(x_n)=(\frac{1}{2}, \frac{1}{2}, \frac{1}{4}, \frac{1}{8}, \frac{1}{16}, \ldots)$. Hence $A(x_n)=[0,\frac{3}{2}]$. Thus $f^{-1}(1)=\Big(([0,\frac{1}{2}]\cup [1,\frac{3}{2}])\setminus D\Big)\cup\{0,\frac{3}{2}\}$, $f^{-1}(2)=\bigg(D\cap \Big((0,\frac{1}{2})\cup (1,\frac{3}{2})\Big)\bigg)\cup \Big((\frac{1}{2},1)\setminus D\Big)$, $f^{-1}(3)=\{\frac{1}{2},1\}$, $f^{-1}(4)=(\frac{1}{2},1)\cap D$.
where $D$ is the set of all dyadic numbers, that is $D=\{\frac{a}{2^b} : a\in\N_0, b\in \N_0\}$. Thus $f^{-1}(\mathfrak{c})=\emptyset$ and $f^{-1}(1)$ is not dense in $A(x_n)$.
\end{example}

\begin{problem}
Note that if we repeat more but still finitely many terms as in the Example \ref{jednopowtorzenie}, then we get $f^{-1}(\mathfrak{c})=\emptyset$ and the length of the gap for $f^{-1}(1)$ in the middle of the achievement set will increase. Finally when we repeat infinitely many of the terms, then the $\mathfrak{c}$ will appear and $f^{-1}(1)$ will be empty. Hence we may ask is it true that $f^{-1}(\mathfrak{c})\neq\emptyset$ implies that $f^{-1}(1)$ is not dense in any interval of the achievement set
($f^{-1}(1)$ is a meager set).
If the answer is negative, then we may ask if it is possible to construct a sequence with both $f^{-1}(1)$ and $f^{-1}(\mathfrak{c})$ being dense in $E(x_n)$. 
\end{problem}

\begin{problem}
Define some kind of minimality of representation for Cantor sets $E(y_n)$. Note that it can not be defined in the similar way since after removing infinite subsequence $(x_n)\subset (y_n)$ we obtain finite set if $(y_n)\setminus (x_n)$ is finite either a Cantor set, when we leave infinite many terms. If asking about subsets does not work maybe it is better to use supersets, that is $(y_n)$ is a minimal representation of $E(y_n)$, which is a Cantor set if after adding any infinite sequence $(x_n)$ for $(z_n)=(x_n)\cup (y_n)$ we have $E(z_n)$ is not a Cantor set ($E(z_n)$ is Cantorval either finite sum of compact intervals). However we think that this definition leads to minimality of each representation and hence does not work. 
\end{problem}

\section{Cantors with  continuum points}
 It is almost obvious that for a multigeometric series with $1$ to $1$ representations in $\Sigma=\{\sigma_0<\sigma_1<\ldots<\sigma_{k}\}$ for $q$ small enough we have $U(c_1,c_2,\ldots,c_m;q)=E(c_1,c_2,\ldots,c_m;q)$. If $\delta=\min\{\sigma_{i}-\sigma_{i-1} : i\in\{1,\ldots,k\}\}$, then for $\frac{\sigma_{k} q}{1-q}<\delta$ (or $q\in(0,\frac{\delta}{\delta+\sigma_{k}})$), $(c_1,c_2,\ldots,c_m;q)$ is quickly convergent. Quick convergence suffices for the equality $U(x_n)=E(x_n)$. On the other hand it is not hard to construct a Cantor set with  $\mathfrak{c}$ - point: it has been already done in many papers. Most of the methods are based on repeating the terms. In \cite{Nymannlin} the author shows how to construct a Cantor set, which algebraic sum with itself remains a Cantor set. The sums of more than two sets were also considered. Moreover  in \cite{BBGS} the authors considered multigeometric series and showed that if $q<\frac{1}{\vert\Sigma\vert}$, then $E(c_1,c_2,\ldots,c_m;q)=\{\sum_{n=1}^{\infty}y_nq^n : (y_n)\in\Sigma^{\infty}\}$ is a Cantor set, no matter if the elements $c_1, c_2,\ldots, c_m$ repeat. 
 
\begin{example}\label{samejedynki}
Let $E=E(\underbrace{1,1,\ldots,1}_{m - \ \text{times}};q)$. 
\begin{enumerate}
\item for $q<\frac{1}{m+1}$ the multigeometric sequence is quickly convergent, so the set $E$ is a Cantor set;
\item for $q\geq\frac{1}{m+1}$  the multigeometric sequence is slowly convergent, so the set $E$ is an interval
\end{enumerate}
\end{example}
Note that for the multigeometric sequence considered in Example \ref{samejedynki} we have $\frac{1}{m+1}=\frac{1}{\vert \Sigma\vert}=\frac{\delta}{\delta+\sigma_{k}}$. In general if $c_1,\ldots,c_m$ is a sequence of natural numbers, then $\sigma_{k}\geq\delta\cdot(\vert\Sigma\vert-1)$. Hence $\frac{\delta}{\delta+\sigma_{k}}\leq \frac{1}{\vert \Sigma\vert}$. Moreover the equality  $\frac{\delta}{\delta+\sigma_{k}}= \frac{1}{\vert \Sigma\vert}$ holds iff the sequence $c_1,\ldots, c_m$ is constant.  

\begin{theorem}
Let $c_1,\ldots, c_m$ be a non-constant sequence of natural numbers. Then there exists 
a sequence $(q_p)$ of ratios such that $\lim_{p\to\infty}q_p=\frac{\delta}{\delta+\sigma_{k}}$ and  $E(c_1,c_2,\ldots,c_m;q_p)$ contains $\mathfrak{c}$ - point for every $p\in\mathbb{N}$.
\begin{proof}
Let $\delta=\sigma_j-\sigma_{j-1}=\min\{\sigma_{i}-\sigma_{i-1} : i\in\{1,\ldots,k\}\}$. Fix $p\in\mathbb{N}$. Let $q_p$ satisfies $\delta=\sigma_k(q_p+q_p^2+\ldots+q_p^{p})$. 
Clearly $\lim_{p\to\infty}q_p=\frac{\delta}{\delta+\sigma_{k}}$. Let us consider $x=\sigma_j q_p+\sigma_j q_p^{p+2}+\sigma_j q_p^{2p+3}+\ldots=\sigma_j\sum_{n=0}^{\infty}q_p^{np+n+1}$. Each summand $\sigma_jq_p^{np+n+1}$ can be replaced by $\sigma_{j-1}q_p^{np+n+1}+\sigma_k(q_p^{np+n+2}+q_p^{np+n+3}+\ldots+q_p^{(n+1)p+n+1})$. Hence $x$ has $\mathfrak{c}$ - many expansions.
\end{proof}
\end{theorem}

\section{Slim representations}
We say that $(x_n)$ is slim, if $E(x_n)$ has no $\mathfrak{c}$-points. If there exists a slim sequence $(y_n)$ with $E(y_n)=E(x_n)$, then we say that $E(x_n)$ has a slim representation. 

\begin{definition}
A monotonic sequence $(x_n)$ with positive terms converging to zero is called semi-fast convergent if it satisfies the condition
\[
x_n>\sum_{k:x_k<x_n}x_k\text{ for all }n\in\N
\]
where the sum is over all indexes $k$ such that $x_k<x_n$
\end{definition}
Properties of semi-fast convergent series were deeply investigated in \cite{BFPW2} and \cite{PWT}. We are following some ideas found there. 

Let $\sum x_n$ be a semi-fast convergent series. Then there are two sequences $(\alpha_k)$ - strictly decreasing tending to zero and $(N_k)$ - sequence of natural numbers such that
\[
(x_n)=(\underbrace{\alpha_1,\alpha_1,\dots,\alpha_1}_{N_1}, \underbrace{\alpha_2,\alpha_2,\dots,\alpha_2}_{N_2}, \underbrace{\alpha_3,\alpha_3,\dots,\alpha_3}_{N_3}, \alpha_4,\dots)
\]
and $\alpha_k>\sum_{i>k}N_i\alpha_i$ for every $k\in\N$. We will use the notation $(\alpha_k,N_k)$ for semi-fast sequences (or series). 

\begin{theorem}\label{TheoremSemiFast}
A semi-fast convergent sequence $(\alpha_k,N_k)$ has a unique (slim) representation if and only if every (for almost all $k$) $N_k$ is of the form $2^{n_k}-1$, where $n_k\in\N$.  
\end{theorem}

\begin{proof}Let $(x_n)=(\alpha_k,N_k)$ and assume that $E(y_m)=E(\alpha_k,N_k)$, in other words $(y_m)$ is a representation of $(\alpha_k,N_k)$. By $r_k$ we denote the $k$-th remainder of $\sum_{n=1}^\infty x_n$, that is $r_k=\sum_{i>k}x_i$. Note that $r_{N_1}=\sum_{i>1}N_i\alpha_i$. Then $(r_{N_1},\alpha_1)$ is the largest gap from the left. Moreover $(\alpha_1+r_{N_1},2\alpha_1), (2\alpha_1+r_{N_1},3\alpha_1),\dots, ((N_1-1)\alpha_1+r_{N_1},N_1\alpha_1)$ are gaps as well. By the Second Gap Lemma there is a $p_1\in\N$ such that 
\[
\{0,\alpha_1,2\alpha_1,3\alpha_1,\dots,N_1\alpha_1\}=\{\sum_{i=1}^{p_1}\varepsilon_i y_i:\varepsilon\in\{0,1\}^{p_1}\}.
\]
Then by the Second Gap Lemma $y_{p_1}=\alpha_1$. Note that $y_i$, for $i\leq p_1$, are multipliers of $\alpha_1$. Note that $\{0,1\}^{p_1}\ni\varepsilon\mapsto\sum_{i=1}^{p_1}\varepsilon_i y_i$ is one-to-one if and only if the set $\{\sum_{i=1}^{p_1}\varepsilon_i y_i:\varepsilon\in\{0,1\}^{p_1}\}$ has exactly $2^{p_1}$ elements, which in turn is equivalent to $N_1=2^{p_1}-1$. Proceeding inductively we find an increasing sequence $(p_k)$ of indexes such that
\[
\{0,\alpha_k,2\alpha_k,3\alpha_k,\dots,N_k\alpha_k\}=\{\sum_{i=p_{k-1}+1}^{p_k}\varepsilon_i y_i:\varepsilon\in\{0,1\}^{(p_{k-1},p_k]}\}\text{, }y_{p_k}=\alpha_k\text{ and }\sum_{i=p_{k-1}+1}^\infty y_i<y_{p_k}.
\]
Moreover, $\{0,1\}^{p_k}\ni\varepsilon\mapsto\sum_{i=1}^{p_k}\varepsilon_iy_i$ is one-to-one if and only if $N_k=2^{p_k}-1$.

If there is $k$ such that $N_k\neq 2^{p_k}-1$, then clearly $(\alpha_k,N_k)$ does not have unique representation. If there is an infinite sequence $(k_j)$ with $N_{k_j}\neq 2^{p_{k_j}}-1$ and $(y_m)$ is a representation of $(\alpha_k,N_k)$, then for any $j$ we can find two distinct $\varepsilon^j,\bar{\varepsilon}^j\in\{0,1\}^{(p_{k_j-1},p_{k_j}]}$ such that   
\[
\sum_{i=p_{k_j-1}+1}^{p_{k_j}}\varepsilon^j_i y_i=\sum_{i=p_{k_j-1}+1}^{p_{k_j}}\bar{\varepsilon}^j_i y_i
\]
Consider a point $x=\sum_{j=1}^\infty \sum_{i=p_{k_j-1}+1}^{p_{k_j}}\varepsilon^j_i y_i$. Now, for any set $X\subset\N$, we have
\[
x=\sum_{j\in X} \sum_{i=p_{k_j-1}+1}^{p_{k_j}}\varepsilon^j_i y_i+ \sum_{j\notin X} \sum_{i=p_{k_j-1}+1}^{p_{k_j}}\bar{\varepsilon}^j_i y_i.
\]
Thus the point $x$ has continuum many representations in $E(y_m)$. Therefore $(y_m)$ is not a slim representation of $(\alpha_k,N_k)$, and consequently $(\alpha_k,N_k)$ does not have slim representation. 
\end{proof}

Using Theorem \ref{TheoremSemiFast} and results from \cite{BFPW2} we can extend the list of equivalent conditions.
\begin{corollary}
Let $(\alpha_k,N_k)$ be a semi-fast convergent sequence. The following conditions are equivalent
\begin{itemize}
    \item[(i)] $(\alpha_k,N_k)$ has a unique  representation;
    \item[(ii)] every $N_k$ is of the form $2^{n_k}-1$;
    \item[(iii)] $E(\alpha_k,N_k)$ is a central Cantor set
    \item[(iv)] $E(\alpha_k,N_k)$ has a representation $E(y_n)$ where $(y_n)$ is fast convergent.
\end{itemize}
\end{corollary}

By Theorem \ref{TheoremSemiFast} there is a Cantor set having no slim representation. On the other hand any finite union of intervals \textit{has} slim representation, which is a simple observation. This lead us to the following problem. 
\begin{problem}
Is there a Cantorval having no slim representation?
\end{problem}
Since the Guthrie-Nymann Cantorval $E(3,2;\frac{1}{4})$ has no $\mathfrak{c}$-points, it has slim representation. Moreover, $E(3,2;\frac14)$ is minimal Cantorval in that sense that $E(3,2;q)$ is a Cantor set for every $q<\frac14$. We would like to know if such minimality implies slimness.
\begin{problem}
Assume that $E(a_1,\dots a_n;q)$ is a Cantorval such that for any $q'<q$, $E(a_1,\dots a_n;q')$ is a Cantor set. Is $E(a_1,\dots a_n;q)$ slim? 
\end{problem}

\end{document}